\RequirePackage[l2tabu, orthodox]{nag}

\documentclass[12pt]{amsart}
\usepackage{fullpage,url,amssymb,enumerate}
\usepackage[all]{xy} 
\usepackage{comment}
\usepackage{graphicx}
\usepackage{youngtab, ytableau}

\usepackage[OT2,T1]{fontenc}

\usepackage{amsthm}

\usepackage{color}

\usepackage{tikz}

\def\bC{\mathbb{C}}

\def\bP{\mathbb{P}}

\usepackage{mathdots}

\DeclareMathOperator{\Fl}{Fl}
\DeclareMathOperator{\Gr}{Gr}

\DeclareMathOperator{\pr}{pr}

\DeclareMathOperator{\ssyt}{SSYT}

\newtheorem{thm}{Theorem}
\newtheorem{lem}[thm]{Lemma}
\newtheorem{cor}[thm]{Corollary}
\newtheorem{prop}[thm]{Proposition}

\newtheorem{rem}[thm]{Remark}

\newtheorem{defn}[thm]{Definition}

\makeatletter
\g@addto@macro\bfseries{\boldmath} 
\makeatother

\usepackage{microtype}

\usepackage[
	backref,
	pdfauthor={}, 
	pdftitle={torus_orbit},
]{hyperref}

\begin{document}

\title{Torus orbit closures and 1-strip-less tableaux}

\author{Carl Lian}

\address{Tufts University, Department of Mathematics, 177 College Ave
\hfill \newline\texttt{}
 \indent Medford, MA 02155} \email{{\tt Carl.Lian@tufts.edu}}

\begin{abstract}
We compare two formulas for the class of a generic torus orbit closure in a Grassmannian, due to Klyachko and Berget-Fink. The naturally emerging combinatorial objects are semi-standard fillings we call \emph{1-strip-less tableaux}.
\end{abstract}

\maketitle

\section{Introduction}

Let $\Gr(r,n)$ be the Grassmannian of $r$-planes in $\bC^n$, on which an $n$-dimensional torus $T$ acts in the standard way. Let $x\in \Gr(r,n)$ be a general point, let $X:=\overline{T\cdot x}$ be the closure of the orbit of $x$, and let $[X]\in H^{2(r-1)(n-r-1)}(\Gr(r,n))$ be the corresponding cohomology class. Such classes have received considerable attention in the literature \cite{k,speyer,bf,lpst,nt}; the author's own interest in them stems from their relationship to fixed-domain curve counts in the projective space $\bP^{r-1}$ \cite{l_coll}.

We have the following two formulas for $[X]$ originally due to Klyachko and Berget-Fink, respectively, obtained via different methods.

\begin{thm}\cite[Theorem 6]{k}\label{k_formula}
Write
\begin{equation*}
[X]=\sum_{\mu} \Gamma^\mu_{r,n}\sigma_{\mu}
\end{equation*}
in terms of the basis of Schubert cycles $\sigma_\mu$. Then, we have
\begin{equation*}
\Gamma^\mu_{r,n}=\sum_{j=0}^{m(\mu)}(-1)^j\binom{n}{j}|\ssyt_{r-j}(\mu^j)|,
\end{equation*}
where:
\begin{itemize}
\item $m(\mu)$ denotes the number of parts of $\mu$ equal to $n-r$ (the largest possible), 
\item $\mu^j$ denotes the partition of length $r-j$ obtained from $\mu$ by removing $j\le m(\mu)$ parts of size $n-r$, and
\item $|\ssyt_{r'}(\mu')|$ denotes the number of semi-standard Young tableaux (SSYTs) of shape $\mu'$ filled with entries $1,2,\ldots,r'$.
\end{itemize}
\end{thm}

See \S\ref{conventions} for a detailed discussion of notational conventions for Schubert calculus and related combinatorics.

\begin{thm}\cite[Theorem 5.1]{bf}\label{bf_formula}
We have
\begin{equation*}
[X]=\sum_{\lambda\subset(n-r-1)^{r-1}}\sigma_{\lambda}\sigma_{\widetilde{\lambda}},
\end{equation*}
where, for any partition $\lambda\subset(n-r-1)^{r-1}$, $\widetilde{\lambda}$ denotes its complement inside the rectangle $(n-r-1)^{r-1}$.
\end{thm}

Therefore, for abstract reasons, the following holds.

\begin{thm}\label{main_formula}
We have
\begin{equation*}
\sum_{\lambda\subset(n-r-1)^{r-1}}\sigma_{\lambda}\sigma_{\widetilde{\lambda}}=\sum_{\mu} \Gamma^\mu_{r,n}\sigma_{\mu},
\end{equation*}
where the $\Gamma^\mu_{r,n}$ are as in Theorem \ref{k_formula}.
\end{thm}

The purpose of this paper is to give a direct combinatorial proof of Theorem \ref{main_formula}. Our main calculation applies Coskun's geometric Littlewood-Richardson rule \cite{coskun} to the left-hand side in such a way that the coefficients $\Gamma^\mu_{r,n}$ emerge naturally. In fact, we arrive at the following combinatorial interpretation for the coefficients $\Gamma^\mu_{r,n}$:

\begin{thm}\label{stripless_interpretation}
$\Gamma^\mu_{r,n}$ is equal to the number of \textit{1-strip-less} SSYTs (see \S\ref{yt_section}) of shape $\mu$, filled with entries $1,2,\ldots,r$. 
\end{thm}

In particular, in the alternating sum of Theorem \ref{k_formula}, the first term dominates, and we get a \emph{non-negative} interpretation of the integers $\Gamma^\mu_{r,n}$. In fact, the following version of Theorem \ref{stripless_interpretation} was known to Klyachko \cite[Theorem 3.3]{k2}: the number $\Gamma^\mu_{r,n}$ was shown to count the number of \emph{standard} Young tableaux of shape $\overline{\lambda}$, the complement of $\lambda$ inside the rectangle $(n-r)^r$, with exactly $(r-1)$ descents. Such objects are naturally in bijection with 1-strip-less SSYTs; we describe this bijection (explained to us by Philippe Nadeau) in the appendix.

In our calculation, the 1-strip-less SSYTs are the naturally occurring objects. As we will see in \S\ref{LR_sec}, the geometric Littlewood-Richardson rule allows us to express the products $\sigma_{\lambda}\sigma_{\widetilde{\lambda}}$ in terms of simple Schubert classes, see Corollary \ref{explicit_M}. From here, the Pieri rule allows us on the one hand to interpret the coefficients in the Schubert basis as counts of 1-strip-less tableaux in \S\ref{stripless_sec}, and on the other hand to recover Klyachko's formula in \S\ref{klyachko_comparison_sec}.

\subsection{Acknowledgments}
We thank Andrew Berget, Izzet Coskun, Alex Fink, Maria Gillespie, Philippe Nadeau, Andrew Reimer-Berg, Hunter Spink, and Vasu Tewari for helpful comments and conversations, as well as the referees for their comments and for encouraging the author to include additional details. During the preparation of this article, the author received support from the MATH+ incubator grant ``Tevelev degrees.'' 


\section{Preliminaries}\label{conventions}

\subsection{Young tableaux}\label{yt_section}

Let $\lambda=(\lambda_1,\ldots,\lambda_r)$ be a partition. Our convention throughout is that the parts of $\lambda$ are non-increasing $(\lambda_1\ge\cdots\ge\lambda_r\ge0)$, and furthermore that $\lambda$ is defined implicitly with respect to an integer $n$ for which $\lambda_1\le n-r$ (equivalently, $\lambda\subset(n-r)^r$). The Young diagram of $\lambda$ is taken to be left- and upward-justified, with larger parts depicted on top. 

The \emph{complement} of $\lambda$, denoted $\overline{\lambda}$, is the partition $(n-r-\lambda_r,\ldots,n-r-\lambda_1)$. Pictorially, the complement $\overline{\lambda}$ is obtained as the (rotated) complement of $\lambda$ inside the rectangle $(n-r)^r$. Note that we distinguish the notation $\overline{\lambda}$ and the notation $\widetilde{\lambda}$ appearing in Theorem \ref{bf_formula}.

Below, the Young diagram of the partition $\lambda=(10,9,4,2)$, defined with respect to $n=14$ is shown. Its complement (before rotation) $\overline{\lambda}=(8,6,1,0)$ is shaded in gray.

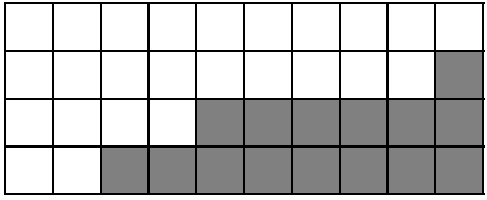
\begin{equation*}
\ydiagram[*(white) ]
{10,9,4,2}
*[*(gray)]{10,10,10,10}
\end{equation*}

A \emph{strip} $S$ of $\lambda$ is a subset of $n-r$ boxes in the Young diagram of $\lambda$ satisfying the following two properties:
\begin{itemize}
\item  No two boxes of $S$ lie in the same column.
\item Given any distinct boxes $b_1,b_2$ of $S$, if $b_1$ lies in a column to the left of $b_2$, then $b_1$ does not also lie in a row above $b_2$.
\end{itemize}
An example of a strip in the partition $\lambda=(10,9,4,2)$, where we take $n=14$, is shaded below. However, the same set of boxes is \emph{not} a strip if $\lambda$ is defined with respect to $n>14$.
%
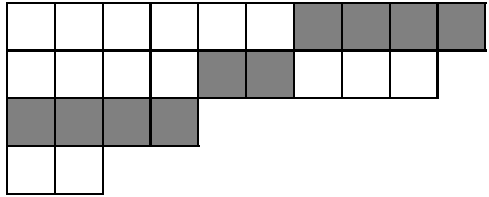
\begin{equation*}
\ydiagram[*(white) ]
{6,4}
*[*(gray)]{6+4,4+2,4}
*[*(white)]{10+0,6+3,4+0,2}
\end{equation*}

A \emph{semi-standard Young tableau (SSYT)} of shape $\lambda$ is a filling of the boxes of $\lambda$ with the entries $1,2,\ldots,r$ so that entries increase weakly across rows and strictly down columns. We will often abuse notation, using the letter $\lambda$ to denote either a partition or a SSYT of that shape. The number of SSYTs of shape $\lambda$ is denoted $|\ssyt_r(\lambda)|$, and is given by the formula
\begin{equation*}
|\ssyt_r(\lambda)|=s_\lambda(1^r)=\prod_{u\in\lambda}\frac{r+c(u)}{h(u)},
\end{equation*}
see \cite[Corollary 7.21.4]{EC2}. Here, $s_\lambda(1^r)$ denotes the Schur function associated to $\lambda$, specialized so that the first $r$ variables are equal to 1 and all others are equal to zero. In the last formula, the product is over all boxes $u$, and if $u$ is in the $i$-th row and $j$-th column of $\lambda$, then by definition, $c(u)=j-i$ and $h(u)$ is the total number of boxes either to the right (and in the same row) of, below (and in the same column), or equal to $u$.

The \emph{$k$-weight} $w_k(\lambda)$ of a SSYT $\lambda$ is the number of appearances of the entry $k$. The \emph{type} of $\lambda$ is the tuple $(w_1(\lambda),w_2(\lambda),\ldots,w_r(\lambda))$.

For $i=1,2,\ldots,r$, an \emph{$(i)$-strip} of a SSYT is a strip, all of whose boxes are filled with the entry $i$. A \emph{0-strip} (written without parentheses) is, by definition, an $(i)$-strip for some $i$. A SSYT is \emph{0-strip-less} if it contains no 0-strips. Below, the SSYT of shape $\lambda=(10,9,4,2)$ (with $n=14$) has a $(3)$-strip and no other 0-strips. If we take instead $n>14$, then the same SYT is 0-strip-less.
\begin{equation*}
\begin{ytableau}
1 & 1 & 1 & 1 & 1 & 2 & 3 & 3 & 3 & 3 \\
2 & 2 & 2 & 2 & 3 & 3 & 4 & 4 & 4\\
3 & 3 & 3 & 3\\
4 & 4
\end{ytableau}
\end{equation*}

Let $|\ssyt^0_r(\lambda)|$ denote the number of 0-strip-less SSYTs of shape $\lambda$ (with entries $1,2,\ldots,r$).
\begin{lem}\label{count_0strip}
Given a partition $\lambda\subset(n-r)^r$, adopt the notation $m(\lambda)$ and $\lambda^j$, for $j=0,1,\ldots,m(\lambda)$, of the statement of Theorem \ref{k_formula}. Then, we have
\begin{equation*}
|\ssyt^0_r(\lambda)|=\sum_{j=0}^{m(\lambda)}(-1)^j\binom{r}{j}|\ssyt_{r-j}(\lambda^j)|.
\end{equation*}
\end{lem}
\begin{proof}
Given any $j$-element subset of $S\subset \{1,2,\ldots,r\}$, the number $|\ssyt_{r-j}(\mu^j)|$ counts SSYTs of shape $\mu$ with an $(s)$-strip for any $s\in S$. Indeed, deleting all boxes containing an entry $s\in S$ in such a SSYT yields (after appropriate re-shifting) an SSYT of shape $\mu^j$ filled with entries lying in $\{1,2,\ldots,r\}-S$; this is easily seen to be a bijection. The factor $\binom{r}{j}$ enumerates subsets $S$ of size $j$. The lemma now follows from Inclusion-Exclusion.
\end{proof}

Similarly, for $i=1,2,\ldots,r-1$, an \emph{$(i,i+1)$-strip} of a SSYT is a strip, all of whose boxes are filled with the entry $i$ or $i+1$, and for which all instances of $i$ all appear to the left of all instances of $i+1$. By convention, an $(i)$-strip is both an $(i,i+1)$- and an $(i-1,i)$-strip. A \emph{1-strip} (written without parentheses) is, by definition, an $(i,i+1)$-strip for some $i$. A SSYT is \emph{1-strip-less} if it contains no 1-strips (and therefore no 0-strips). Note that an SSYT of shape $\lambda$ with $\lambda_1<n-r$ is automatically 1-strip-less. Below, the SSYT of shape $\lambda=(10,9,4,2)$ has a unique $(2,3)$-strip and no other 1-strips.
\begin{equation*}
\begin{ytableau}
1 & 1 & 1 & 1 & 2 & 2 & 3 & 3 & 3 & 3 \\
2 & 2 & 2 & 2 & 3 & 4 & 4 & 4 & 4\\
3 & 3 & 3 & 4\\
4 & 4
\end{ytableau}
\end{equation*}

\subsection{Schubert calculus}\label{schubert_sec}

Let $W$ be a vector space of dimension $n$, assumed over $\bC$ for concreteness. Then, $\Gr(r,W)\cong\Gr(r,n)$ is the Grassmannian of $r$-dimensional subspaces of $W$.

Fix a complete flag $F$ of subspaces
\begin{equation*}
0=F_0\subset F_1\subset\cdots \subset F_{n}=W.
\end{equation*}
Let $\lambda=(\lambda_1,\ldots,\lambda_r)$ be a partition, with $\lambda_1\le n-r$. Then, the Schubert variety $\Sigma^F_{\lambda}\subset\Gr(r,W)$ is by definition the closed subvariety of (complex) codimension $|\lambda|$ consisting of subspaces $V\subset W$ of dimension $r$ for which
\begin{equation*}
\dim(V\cap F_{n-r+i-\lambda_{i}})\ge i
\end{equation*}
for $i=1,\ldots,r$. The class of $\Sigma^F_{\lambda}$ in $H^{2|\lambda|}(\Gr(r,W))$ is denoted $\sigma_\lambda$.

\section{Application of the Littlewood-Richardson rule}\label{LR_sec}

In this section, we express the products $\sigma_{\lambda}\sigma_{\widetilde{\lambda}}$ appearing on the left hand side of Theorem \ref{main_formula} in terms of simple Schubert classes. The main result is Corollary \ref{explicit_M}. We apply Coskun's Littlewood-Richardson rule \cite{coskun}; we will only need some aspects of the algorithm, which we review as we need them. Other versions of the Littlewood-Richardson rule can presumably prove the same formula; our presentation is largely idiosyncratic.

We first carry out \cite[Algorithm 3.19]{coskun} to express the class $\sigma_{\lambda}\sigma_{\widetilde{\lambda}}$ in terms of a \emph{Mondrian tableau}. Geometrically, this will interpret $\sigma_{\lambda}\sigma_{\widetilde{\lambda}}$ as the class of a subvariety of $\Gr(r,n)$ defined generically by the condition that $V\subset\bC^n$ contain a basis of vectors $v_1,\ldots,v_r$, where each $v_j$ is constrained to lie in a particular subspace of $\bC^n$.

Write $\lambda=(\lambda_1,\ldots,\lambda_{r-1},0)$, so that $\widetilde{\lambda}=(n-r-1-\lambda_{r-1},\ldots,n-r-1-\lambda_1,0)$. We also write $\lambda_r=0$ and $\lambda_0=n-r-1$. Fix a basis $e_1,\ldots,e_n$ of $\bC^n$. Then, $\sigma_{\lambda}$ is the class of the Schubert variety of $V\in\Gr(r,n)$ satisfying
\begin{equation*}
\dim(V\cap\langle e_1,\ldots,e_{n-r-\lambda_j+j}\rangle)\ge j
\end{equation*}
for $j=1,\ldots,r$, and $\sigma_{\widetilde{\lambda}}$ is the class of the Schubert variety of $V\in\Gr(r,n)$ satisfying
\begin{equation*}
\dim(V\cap\langle e_{n-r-\lambda_{j-1}+j-1},\ldots,e_{n}\rangle)\ge r+1-j
\end{equation*}
for $j=1,\ldots,r$.

Because these two Schubert varieties are defined with respect to the transverse flags
\begin{align*}
0\subset\langle e_1\rangle\subset\langle e_1,e_{2}\rangle\subset\cdots\subset\bC^n,\\
0\subset\langle e_{n}\rangle\subset\langle e_{n-1},e_{n}\rangle\subset\cdots\subset\bC^n,
\end{align*}
the product $\sigma_{\lambda}\sigma_{\widetilde{\lambda}}$ is represented by the generically transverse intersection of these two Schubert varieties. Combining the two conditions above for each $j$, we obtain the $r$ conditions
\begin{equation*}
\dim(V\cap\langle e_{n-r-\lambda_{j-1}+j-1},\ldots,e_{n-r-\lambda_j+j}\rangle)\ge 1.
\end{equation*}
for $j=1,\ldots,r$, as subspaces of $V$ of dimensions $j,r+1-j$ must intersect non-trivially. By \cite[Theorem 3.21]{coskun} and its proof, these conditions suffice to understand the class $\sigma_\lambda\sigma_{\widetilde{\lambda}}$. More precisely:
\begin{prop}
Let $Z$ be the closure on $\Gr(r,n)$ of the locus of subspaces $V\subset\bC^n$ containing a basis $v_1,\ldots,v_r$ with 
\begin{equation*}
v_j\in \langle e_{a_{j-1}},\ldots,e_{a_j}\rangle,
\end{equation*}
$j=1,2,\ldots,r$. Here, we write $a_j=n-r-\lambda_j+j$ for $j=1,2,\ldots,r$, in addition to $a_0=1$.

Then, the class of $Z$ in $H^{2(r-1)(n-r-1)}(\Gr(r,n))$ is equal to $\sigma_\lambda\sigma_{\widetilde{\lambda}}$.
\end{prop}
The subvariety $Z$ is represented by the \emph{Mondrian tableau} $M$ depicted in Figure \ref{mondrian1} in the case $r=4$.
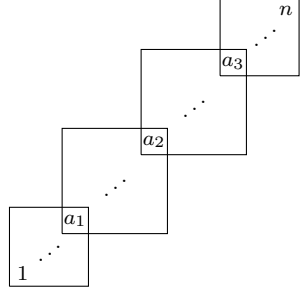
\begin{figure}
\caption{The \emph{Mondrian tableau} associated to the class $\sigma_\lambda\sigma_{\widetilde{\lambda}}$ when $r=4$.}\label{mondrian1}
\begin{tikzpicture} [xscale=0.35,yscale=0.35]
                \node at (0.5,0.5) {\tiny $1$};
                
               \node at (1.5,1.5) {\tiny $\iddots$};
               
                \node at (2.5,2.5) {\tiny $a_1$};
                
		 \node at (4,4) {\tiny $\iddots$};
                
                \node at (5.5,5.5) {\tiny $a_2$};
                
                		 \node at (7,7) {\tiny $\iddots$};
                
                \node at (8.5,8.5) {\tiny $a_3$};
                
          \node at (9.7,9.7) {\tiny $\iddots$};
                
                \node at (10.5,10.5) {\tiny $n$};
                
\draw (0,0) -- (3,0) -- (3,3) -- (0,3) -- (0,0);
\draw (2,2) -- (6,2) -- (6,6) -- (2,6) -- (2,2);
\draw (5,5) -- (9,5) -- (9,9) -- (5,9) -- (5,5);
\draw (8,8) -- (11,8) -- (11,11) -- (8,11) -- (8,8);
\end{tikzpicture}
\end{figure}
The basis elements $e_{1},\ldots,e_{n}$ (represented just by the subscripts for brevity) are recorded along on the diagonal, and the $r$ subspaces $\langle e_{a_{j-1}},\ldots,e_{a_j}\rangle$ are represented by squares $M_j$ containing the corresponding indices. Note that the $M_j$ are pairwise disjoint except for the intersections $M_j\cap M_{j+1}=\{a_j\}$. 

To compute the class associated to $M$ in $H^{2(r-1)(n-r-1)}(\Gr(r,n))$, we will (implicitly) compare it to that of a Mondrian tableau $M'$ with \emph{disjoint} squares $M'_j$ of the same sizes as the $M_j$. The disjointness implies that the $r$ squares impose \emph{independent} conditions on $V$, and the class associated to $M'$ will be given simply by the Pieri rule. However, for such an $M'$ to exist, one needs to enlarge the ambient vector space in which the subspaces live. 

We therefore pass from $\bC^n=\langle e_1,\ldots,e_n\rangle$ to $\bC^{n+r-1}=\langle e_{-(r-2)},\ldots,e_n\rangle$. Abusing notation, we let $M$ denote the Mondrian tableau consisting of the already-defined squares $M_j$, but now with the extra $r-1$ basis elements added in the southwest, and $Z$ denote the closure of the locus on $\Gr(r,n+r-1)$ of subspaces $V=\langle v_1,\ldots,v_r\rangle$ subject to the same conditions $v_j\in \langle e_{a_{j-1}},\ldots,e_{a_j}\rangle$ as before.

The Mondrian tableaux $M,M'$, now defining classes on $\Gr(r,n+r-1)$, are depicted in Figure \ref{mondrian2}.
\begin{figure}
\caption{Comparison of the Mondrian tableaux $M,M'$ when $r=4$. Here, $r-1=3$ basis elements have been added in the southwest corner.}\label{mondrian2}
\begin{tikzpicture} [xscale=0.35,yscale=0.35]
			
			\node at (-2.5,-2.5) {\tiny $-2$};		
		\node at (-1.5,-1.5) {\tiny $-1$};		
	                 \node at (-0.5,-0.5) {\tiny $0$};		
                \node at (0.5,0.5) {\tiny $1$};
                
               \node at (1.5,1.5) {\tiny $\iddots$};
               
                \node at (2.5,2.5) {\tiny $a_1$};
                
		 \node at (4,4) {\tiny $\iddots$};
                
                \node at (5.5,5.5) {\tiny $a_2$};
                
                		 \node at (7,7) {\tiny $\iddots$};
                
                \node at (8.5,8.5) {\tiny $a_3$};
                
          \node at (9.7,9.7) {\tiny $\iddots$};
                
                \node at (10.5,10.5) {\tiny $n$};
                
\draw (0,0) -- (3,0) -- (3,3) -- (0,3) -- (0,0);
\draw (2,2) -- (6,2) -- (6,6) -- (2,6) -- (2,2);
\draw (5,5) -- (9,5) -- (9,9) -- (5,9) -- (5,5);
\draw (8,8) -- (11,8) -- (11,11) -- (8,11) -- (8,8);

\node at (5.5,-5) {$M$};

\end{tikzpicture}
\begin{tikzpicture} [xscale=0.35,yscale=0.35]
                \node at (-2.5,-2.5) {\tiny $-2$};

                \node at (1.5,0.5) {\tiny $a_1-2$};
                
                                \node at (5.5,4.5) {\tiny $a_2-1$};

                \node at (8.7,8.5) {\tiny $a_3$};

                \node at (10.5,10.5) {\tiny $n$};
                
                               \node at (-1,-1) {\tiny $\iddots$};
                             \node at (2.5,2.5) {\tiny $\iddots$};
                              \node at (6.5,6.5) {\tiny $\iddots$};
                             \node at (9.5,9.5) {\tiny $\iddots$};
                
\draw (-3,-3) -- (0,-3) -- (0,0) -- (-3,0) -- (-3,-3);
\draw (0,0) -- (4,0) -- (4,4) -- (0,4) -- (0,0);
\draw (4,4) -- (8,4) -- (8,8) -- (4,8) -- (4,4);
\draw (8,8) -- (11,8) -- (11,11) -- (8,11) -- (8,8);

\node at (5.5,-5) {$M'$};
\end{tikzpicture}
\end{figure}
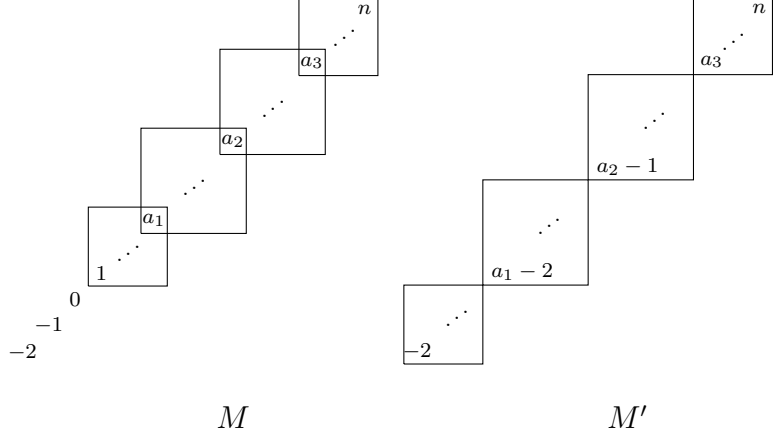

\begin{defn}
Define the class $M(a_1,\ldots,a_{r-1};n)\in H^{2(r-1)(n-1)}(\Gr(r,n+r-1))$ to be the class associated to the Mondrian tableau $M$, where we have added the basis elements $e_{-(r-2)},\ldots,e_0$.

Equivalently, the class $M(a_1,\ldots,a_{r-1};n)$ is equal to the pushforward of $\sigma_{\lambda}\sigma_{\widetilde{\lambda}}$ from $\Gr(r,n)$ to $\Gr(r,n+r-1)$ under the natural inclusion $\langle e_{1},\ldots,e_n\rangle\to \langle e_{-(r-2)},\ldots,e_n\rangle$.
\end{defn}

As the pushforward from $\Gr(r,n)$ to $\Gr(r,n+r-1)$ is injective, the classes $M(a_1,\ldots,a_{r-1};n)$ determine the classes $\sigma_{\lambda}\sigma_{\widetilde{\lambda}}$, simply by replacing the terms $\sigma_\mu$ appearing in the expansion of $M(a_1,\ldots,a_{r-1};n)$ in terms of the Schubert basis with $\sigma_{\mu-(r-1)^r}$.

As we will see, we will not need to deal explicitly with $M'$, but rather, we will understand $M$ inductively, shifting the square $M_1$ to the southwest so that it becomes disjoint from the union of the remaining squares. The squares $M_2,\ldots,M_r$ themselves form a Mondrian tableau of the same form as $M$, but with one fewer square in total, therefore corresponding to conditions on a $(r-1)$-dimensional subspace. The key relation is the following:

\begin{prop}\label{one_step}
Fix integers $a_1,\ldots,a_{r-1}$ with $1<a_1<\cdots<a_{r-1}<n$. Then, we have the following equality on $H^{2(r-1)(n-1)}(\Gr(r,n+r-1))$:
\begin{align*}
&\sigma_{n-a_1}\cdot [\sigma_{(a_1-1)^{r-1}}\cdot M(a_2-a_1+1,\ldots,a_{r-1}-a_1+1;n-a_1+1)]\\
=\quad &M(a_1,\ldots,a_{r-1};n)+\sigma_{n-1}\cdot M(a_2,\ldots,a_{r-1};n).
\end{align*}
\end{prop}

We explain how the terms in the equality are to be interpreted. Write 
\begin{equation*}
\theta:H^{*}(\Gr(r-1,n-r-2))\to H^{*}(\Gr(r,n-r-1))
\end{equation*}
for the homomorphism map of abelian groups defined by $\theta(\sigma_\mu)=\sigma_{(\mu,0)}$, where we have simply appended a part of size 0 to the end of $\mu$.
\begin{itemize}
\item The class $M(a_2-a_1+1,\ldots,a_{r-1}-a_1+1;n-a_1+1)$ is naturally a class of degree $2(r-2)(n-a_1)$ on $\Gr(r-1,n+r-1-a_1)$, therefore a linear combination of Schubert cycles $\sigma_\mu$ with $\mu\subset(n-a_1)^{r-1}$ and $|\mu|=(r-2)(n-a_1)$. The operation ``$\sigma_{(a_1-1)^{r-1}}\cdot-$'' pushes this class forward to $\Gr(r-1,n+r-2)$ by adding $(a_1-1)$ to each part of $\mu$. This class is then mapped to $\Gr(r,n+r-1)$ via the homomorphism $\theta$ (which we have suppressed in the formula of Proposition \ref{one_step}). Finally, the resulting class on $\Gr(r,n+r-1)$ is multiplied by $\sigma_{n-a_1}$.

\item The class $M(a_2,\ldots,a_{r-1};n)$ is naturally a class of degree $2(r-2)(n-1)$ on $\Gr(r-1,n+r-2)$. This class is mapped to a class on $\Gr(r,n+r-1)$ under $\theta$, and then multiplied with the class $\sigma_{n-1}$. Alternatively, one can combine these two steps simply by replacing each term $\sigma_{\mu}$ appearing in the expansion of $M(a_2,\ldots,a_{r-1};n)$ with the partition $\sigma_{(n-1,\mu)}$, obtained by appending a part of maximal length $n-1$ to $\mu$.
\end{itemize}


\begin{proof}
Let $M_{\circ}$ be the Mondrian tableau obtained from $M$ by shifting the square $M_1$ to the southwest by one unit. Geometrically, the condition that $V$ contain a basis element $v_1$ in $\langle e_1,\ldots,e_{a_1}\rangle$ is replaced by the condition that $v_1\in\langle e_0,\ldots,e_{a_1-1}\rangle$; the conditions on the other $r-1$ basis elements remain the same. Let $Z_\circ\subset \Gr(r,n+r-1)$ be the corresponding closed subvariety, that is, the closure of the locus of $V=\langle v_1,\ldots,v_r\rangle$ with $v_1\in\langle e_0,\ldots,e_{a_1-1}\rangle$ and $v_j\in\langle e_{a_j},\ldots,e_{a_{j+1}}\rangle$ for $j\ge2$.

Let $M_+$ be the Mondrian tableau obtained from $M$ by replacing the square $M_1$ with the square of width 1 containing $e_{a_1}$, and $M_2$ with the square containing $e_0,\ldots,e_{a_2}$. In particular, $M_1$ is contained in $M_2$. Geometrically, the condition that $V$ contain two distinct basis elements $v_1\in\langle e_1,\ldots,e_{a_1}\rangle$ and $v_2\in\langle e_{a_1+1},\ldots,e_{a_2}\rangle$ is replaced by the condition that $v_1$ be a non-zero multiple of $e_{a_1}$, and $v_2\in \langle e_0,\ldots,e_{a_2}\rangle$. (The other $r-2$ conditions stay the same.) Let $Z_+\subset \Gr(r,n+r-1))$ be the closed subvariety defined by these new conditions.
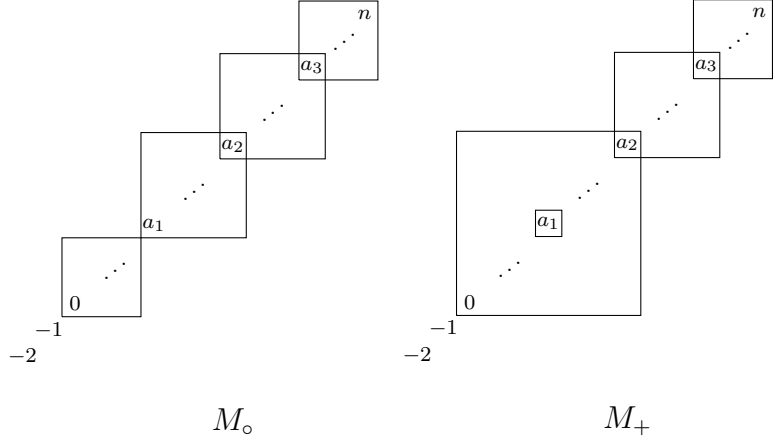
\begin{figure}
\caption{The Mondrian tableaux $M_\circ,M_+$ when $r=4$. After one first step of the geometric Littlewood-Richardson rule is applied, to $M_\circ$, one obtains $M_+$ and $M$ (Figure \ref{mondrian2})}\label{mondrian3}
\begin{tikzpicture} [xscale=0.35,yscale=0.35]
			
			\node at (-2.5,-2.5) {\tiny $-2$};		
		\node at (-1.5,-1.5) {\tiny $-1$};		
	                 \node at (-0.5,-0.5) {\tiny $0$};

               \node at (1,1) {\tiny $\iddots$};
               
                \node at (2.5,2.5) {\tiny $a_1$};
                
		 \node at (4,4) {\tiny $\iddots$};
                
                \node at (5.5,5.5) {\tiny $a_2$};
                
                		 \node at (7,7) {\tiny $\iddots$};
                
                \node at (8.5,8.5) {\tiny $a_3$};
                
          \node at (9.7,9.7) {\tiny $\iddots$};
                
                \node at (10.5,10.5) {\tiny $n$};
                
\draw (-1,-1) -- (2,-1) -- (2,2) -- (-1,2) -- (-1,-1);
\draw (2,2) -- (6,2) -- (6,6) -- (2,6) -- (2,2);
\draw (5,5) -- (9,5) -- (9,9) -- (5,9) -- (5,5);
\draw (8,8) -- (11,8) -- (11,11) -- (8,11) -- (8,8);

\node at (5.5,-5) {$M_\circ$};

\end{tikzpicture}
\begin{tikzpicture} [xscale=0.35,yscale=0.35]
			
			\node at (-2.5,-2.5) {\tiny $-2$};		
		\node at (-1.5,-1.5) {\tiny $-1$};		
	                 \node at (-0.5,-0.5) {\tiny $0$};

               \node at (1,1) {\tiny $\iddots$};
               
                \node at (2.5,2.5) {\tiny $a_1$};
                
		 \node at (4,4) {\tiny $\iddots$};
                
                \node at (5.5,5.5) {\tiny $a_2$};
                
                		 \node at (7,7) {\tiny $\iddots$};
                
                \node at (8.5,8.5) {\tiny $a_3$};
                
          \node at (9.7,9.7) {\tiny $\iddots$};
                
                \node at (10.5,10.5) {\tiny $n$};
                
\draw (-1,-1) -- (6,-1) -- (6,6) -- (-1,6) -- (-1,-1);
\draw (2,2) -- (3,2) -- (3,3) -- (2,3) -- (2,2);
\draw (5,5) -- (9,5) -- (9,9) -- (5,9) -- (5,5);
\draw (8,8) -- (11,8) -- (11,11) -- (8,11) -- (8,8);

\node at (5.5,-5) {$M_+$};

\end{tikzpicture}
\end{figure}

The Mondrian tableaux $M_\circ,M_+$ are depicted in Figure \ref{mondrian3}. Their combinatorial significance is the following: if one starts with $M_\circ$, the geometric Littlewood-Richardson Rule \cite[Algorithm 3.12, step GA2]{coskun} proceeds by shifting $M_1$ (the ``active square'') one unit to the northeast, to obtain the Mondrian tableau $M$. However, one also records the data of the Mondrian tableau $M_+$ which is obtained by replacing $M_1$ and its neighbor square $M_2$ with their \emph{new} intersection and \emph{old} union. We obtain as a consequence \cite[Theorem 3.32]{coskun} that
\begin{equation}\label{subvarieties_break}
[Z_\circ]=[Z]+[Z_+]=M(a_1,\ldots,a_{r-1};n)+[Z_+],
\end{equation}
where we use square brackets $[-]$ to denote the cohomology classes corresponding to subvarieties. The geometric content of \eqref{subvarieties_break} is explained in Remark \ref{subvarieties_break_rmk} at the end of this section.

%

It now suffices to identify the classes $[Z_\circ]$ and $[Z_+]$. First, consider $[Z_\circ]$. The conditions corresponding to $M_1$ and $M_2,\ldots,M_r$ are defined with respect to pairwise disjoint basis elements, so are generically transverse on $\Gr(r,n+r-1)$. The condition corresponding to $M_1$ is that $V$ intersect a fixed subspace of codimension $(n+r-1)-a_1$; the corresponding Schubert variety has class $\sigma_{n-a_1}$.

On the other hand, the condition corresponding to the remaining $(r-1)$ squares is that $V$ contain a hyperplane $V'$ satisfying the same conditions as in the definition of the class $M(a_2-a_1+1,\ldots,a_{r-1}-a_1+1;n-a_1+1)$, except that there are $a_1$ extra basis vectors $e_0,\ldots,e_{a_1-1}$, and the indices of the basis vectors $e_{a_1},\ldots,e_n$ are shifted by $a_1-1$. The class of the subvariety $Z'\subset \Gr(r-1,n+r-1)$ consisting generically of $V'=\langle v_2,\ldots,v_r\rangle$ with $v_j\in \langle e_{a_{j-1}},\ldots,a_{a_{j}}\rangle$ is therefore equal to $\sigma_{(a_1)^{r-1}}\cdot M(a_2-a_1+1,\ldots,a_{r-1}-a_1+1;n-a_1+1)$. Explicitly, if we we write the class $M(a_2-a_1+1,\ldots,a_{r-1}-a_1+1;n-a_1+1)$ as
\begin{equation*}
\sum_{\substack{\mu\subset(n-a_1)^{r-1} \\ |\mu|=(r-2)(n-a_1)}} b_\mu\sigma_\mu\in H^{2(r-2)(n-a_1)}(\Gr(r-1,n+r-1-a_1)),
\end{equation*}
then 
\begin{equation*}
[Z']=\sum_{\substack{\mu\subset(n-a_1)^{r-1} \\ |\mu|=(r-2)(n-a_1)}} b_\mu\sigma_{\mu+a_1^{r-1}}\in H^{2[(r-2)(n-a_1)+a_1(r-1)]}(\Gr(r-1,n+r-1)),
\end{equation*}
where the partition $\mu+a_1^{r-1}$ is obtained from $\mu$ by adding $a_1$ full columns, say, to the left.

The resulting locus on $\Gr(r,n+r+1)$ may be understood as follows. Let $\Fl(r-1,r,n+r-1)$ be the partial flag variety parametrizing nested pairs of subspaces $V'\subset V\subset \bC^{n+r-1}$ of dimensions $r-1,r$, respectively. We have a pair of maps
\begin{equation*}
\xymatrix{ & \Fl(r-1,r,n+r-1) \ar[ld]_{\pr_{V'}} \ar[rd]^{\pr_{V}} & \\
\Gr(r-1,n+r-1) & & \Gr(r,n+r-1)}.
\end{equation*}
remembering the two subspaces individually. Then, the subvariety of $\Gr(r,n+r-1)$ cut out by the conditions imposed by $M_2,\ldots,M_r$ is precisely $\pr_V(\pr_{V'}^{-1}Z')$. An elementary calculation (for example, using the definition given in \S\ref{schubert_sec}) shows that, on the level of cohomology, the correspondence $(\pr_V)_{*}\pr_{V'}^{*}$ sends the Schubert class $\sigma_\mu$ to $\theta(\sigma_{\mu-1^{r-1}})$, which is interpreted to be zero when $1^{r-1}\not\subset\mu$. Here, we recall that $\theta$ is the homomorphism of abelian groups defined immediately after the statement of the proposition, given by $\theta(\sigma_\mu)=\sigma_{(\mu,0)}$.

Therefore, we have
\begin{align*}
(\pr_V)_{*}\pr_{V'}^{*}([Z'])&=\sum_{\substack{\mu\subset(n-a_1)^{r-1} \\ |\mu|=(r-2)(n-a_1)}} b_\mu\theta(\sigma_{\mu+(a_1-1)^{r-1}})\\
&=\theta(\sigma_{(a_1-1)^{r-1}}\cdot M(a_2-a_1+1,\ldots,a_{r-1}-a_1+1;n-a_1+1))
\end{align*}
in $H^{2[(r-1)(n-1)-(n-a_1)]}(\Gr(r,n+r-1))$. Therefore, we conclude that
\begin{equation*}
[Z_\circ]=\sigma_{n-a_1}\cdot [\sigma_{(a_1-1)^{r-1}}\cdot M(a_2-a_1+1,\ldots,a_{r-1}-a_1+1;n-a_1+1)],
\end{equation*}
where, as in the statement of the proposition, we suppress the map $\theta$ inside the square brackets.

Similarly, we claim that
\begin{equation*}
[Z_+]=\sigma_{n-1}\cdot M(a_2,\ldots,a_{r-1};n).
\end{equation*}
Deleting the basis element $e_{a_1}$ and the corresponding $1\times1$ box $M_1$ from $M_+$ yields a Mondrian tableau which describes a subvariety $Z''\subset\Gr(r-1,\bC^{n+r-1}/\langle e_{a_1}\rangle)=\Gr(r-1,n+r-2)$ of class $M(a_2,\ldots,a_{r-1};n)$. Then, taking the condition imposed by $M_1$ back into account, the subvariety $Z_+\subset\Gr(r,\bC^{n+r-1})$ is the image of $Z''$ under the map
\begin{equation*}
\Gr(r-1,\bC^{n+r-1}/\langle e_{a_1}\rangle)\to\Gr(r,\bC^{n+r-1})
\end{equation*}
which sends $V''\subset \bC^{n+r-1}/\langle e_{a_1}\rangle$ to $\langle V'',e_{a_1}\rangle\subset \bC^{n+r-1}$. The induced push-forward map on cohomology sends the class $\sigma_{\mu}$ to $\sigma_{(n-1,\mu)}$, as seen by comparing the Schubert conditions imposed by the flag 
\begin{equation*}
0= F''_0\subset F''_1\subset\cdots\subset F''_{n+r-2}=\bC^{n+r-1}/\langle e_{a_1}\rangle
\end{equation*}
to those imposed by the flag 
\begin{equation*}
0\subset \langle F''_0,e_{a_1}\rangle\subset \langle F''_1,e_{a_1}\rangle\subset\cdots\subset \langle F''_{n+r-2},e_{a_1}\rangle=\bC^{n+r-1},
\end{equation*}
as in the definition in \S\ref{schubert_sec}. We conclude the claimed formula for $[Z_+]$, and combining with our formula for $[Z_\circ]$ completes the proof.
\end{proof}

Note that the basis vectors $e_{-(r-2)},\ldots,e_{-1}$ have not yet an played important role. However, we will apply Proposition \ref{one_step} inductively, $(r-1)$ times. On the $j$-th step, implicitly, the squares $M_1,\ldots,M_j$ are shifted one unit to the southwest, which requires one additional basis vector $e_{1-j}$. (As we have already alluded to, we will not need to perform this combinatorial operation explicitly; Proposition \ref{one_step} will be enough to capture this operation formally.) Therefore, $\Gr(r,n+r-1)$ is the most natural place to state the following formula.

\begin{cor}\label{explicit_M}
Fix integers $a_1,\ldots,a_{r-1}$ with $1<a_1<\cdots<a_{r-1}<n$ as before. Write also $a_0=1$ and $a_r=n$. Then, we have
\begin{equation*}
M(a_1,\ldots,a_{r-1};n)=\sum_{0=s_0<s_1<\cdots< s_\ell =r}(-1)^{r-\ell}\prod_{j=1}^{\ell}\sigma_{(n-1)^{s_j-s_{j-1}-1},n-1-(a_{s_j}-a_{s_{j-1}})}
\end{equation*}
on $\Gr(r,n+r-1)$. Here, the sum is over all (ordered) subsets $\{s_1,\ldots,s_{\ell-1}\}\subset\{1,2,\ldots,r-1\}$ of any size.
\end{cor}
Note that the right hand side may also be written as
\begin{equation*}
\sum_{0=s_0<s_1<\cdots< s_\ell =r}(-\sigma_{n-1})^{r-\ell}\prod_{j=1}^{\ell}\sigma_{n-1-(a_{s_j}-a_{s_{j-1}})}.
\end{equation*}
\begin{proof}
We proceed by induction on $r$. When $r=1$, we obtain simply the statement $M(-;n)=1$, which amounts to the fact that the condition that a single vector lie in all of $\bC^n$ is trivial.

For the inductive step, we verify that the claimed formula satisfies Proposition \ref{one_step}. Applying the inductive hypothesis, the term $\sigma_{n-a_1}\cdot [\sigma_{(a_1-1)^{r-1}}\cdot M(a_2-a_1+1,\ldots,a_{r-1}-a_1+1;n-a_1+1)]$ contains all of the terms in the claimed formula for $M(a_1,\ldots,a_{r-1};n)$ where $s_1=1$. Here, the term $(-1)^{r-\ell}$ does not change, because $r,\ell$ both change by 1. On the other hand, the term $\sigma_{n-1}\cdot M(a_2,\ldots,a_{r-1};n)$ contains all of the remaining terms, where $s_1>1$, but with opposite signs ($r$ differs by 1, but $\ell$ does not change). Indeed, the extra factor of $\sigma_{n-1}$ is absorbed into the $j=1$ factor in each summand. This completes the proof.
\end{proof}


\begin{rem}\label{subvarieties_break_rmk}
We explain the geometric content of the formula \eqref{subvarieties_break}. The claim amounts to studying the degeneration of the subvariety $Z_\circ\subset \Gr(r,n+r-1)$ under the degeneration of the basis $e_{-(r-2)},\ldots,e_n$ that replaces $e_0$ with $e_0^t=te_0+(1-t)e_{a_1}$ at time $t$. When $t\neq0$, one has a closed subvariety $Z^t_\circ\subset \Gr(r,n+r-1)$ defined in the same way as $Z_\circ$, now with respect to the basis $e_{-(r-2)},\ldots,e_0^t,\ldots,e_n$. When $t=0$, these vectors are no longer linearly independent, but we may still study the limit $Z^{\lim}_\circ:=\lim_{t\to 0}Z^t_\circ$.

In the limit, the analysis in the proof of \cite[Theorem 3.32]{coskun} shows  that one of two things may happen generically. The first possibility is that the limits of the vectors $v_1,\ldots,v_r$ remain linearly independent, and we simply have the new condition that $v_1$ lie in the new subspace $\langle e_1,\ldots,e_{a_1}\rangle$, equal to the limit of the subspace $\langle e_0,\ldots,e_{a_1-1}\rangle$ as $t\to 0$. In this way, $Z$ appears as a component of the limit of $Z^{\lim}_\circ$. Second, it may happen that the vectors $v_1$ and $v_2$ become linearly dependent, necessarily both constant multiples of the basis vector $e_{a_1}$, which spans the intersection $\langle e_1,\ldots,e_{a_1}\rangle\cap \langle e_{a_1},\ldots,e_{a_2}\rangle$. If this is the case, then the limit of the span $\langle v_1,v_2\rangle$ must also continue to lie in $\langle e_0,\ldots,e_{a_2}\rangle$, as this holds for every $t\neq0$. This exhibits $Z_+$ as the other component of the limit of $Z^{\lim}_\circ$. Both components are shown in \cite{coskun} to appear with multiplicity 1, and one concludes that $[Z_\circ]=[Z^{\lim}_\circ]=[Z]+[Z_+]$.
\end{rem}

\section{1-strip-less tableaux}\label{stripless_sec}

We now consider the individual terms appearing on the right hand side of Corollary \ref{explicit_M}.

\begin{prop}\label{term_stripless}
Consider the expansion of 
\begin{equation*}
\prod_{j=1}^{\ell}\sigma_{(n-1)^{s_j-s_{j-1}-1},n-1-(a_{s_j}-a_{s_{j-1}})}
\end{equation*}
in $H^{*}(\Gr(r,n+r-1))$ as a linear combination of Schubert cycles $\sigma_\mu$, with $\mu\subset(n-1)^{r}$ and $|\mu|=(n-1)(r-1)$.

Then, the coefficient of $\sigma_\mu$ is equal to the number of SSYTs of shape $\mu$ with the properties that:
\begin{itemize}
\item the $i$-weight $w_i(\mu)$ (that is, the number of boxes of $\mu$ filled with the entry $i$) is equal to $n-1-(a_{i}-a_{i-1})$, for $i=1,2,\ldots,r$, and
\item $\mu$ contains an $(i,i+1)$-strip, for all $i\in\{1,2,\ldots,r-1\}-\{s_1,\ldots,s_{\ell-1}\}$.
\end{itemize}
\end{prop}

Before giving the proof, we first give an illustrative example. Take $n=18$, $r=6$, and $(a_0,a_1,a_2,a_3,a_4,a_5,a_6)=(1,3,7,8,10,14,18)$. Take also $\ell=3$ and $(s_0,s_1,s_2,s_3)=(0,2,5,6)$. Then, the class in question is
\begin{equation*}
\sigma_{17,11}\cdot\sigma_{17,17,10}\cdot\sigma_{13}\in H^{2\cdot5\cdot17}(\Gr(6,23)),
\end{equation*}
which we re-write as 
\begin{equation*}
\sigma_{17}\cdot\sigma_{11}\cdot\sigma_{17}\cdot\sigma_{17}\cdot\sigma_{10}\cdot\sigma_{13}.
\end{equation*}

By the Pieri rule, the coefficient of $\sigma_\mu$ in the expansion of this class is equal to the number of SSYTs of shape $\mu$ and type $(17,11,17,17,10,13)$. One such, for $\mu=(17,17,17,16,13,5)$, is shown below:
\begin{equation*}
\begin{ytableau}
1 & 1 & 1 & 1 & 1 & 1 & 1 & 1 & 1 & 1 & 1 & 1 & 1 & 1 & 1 & 1 & 1 \\
2 & 2 & 2 & 2 & 2 & 2 & 2 & 2 & 2 & 2 & 2 & 3 & 3 & 3 & 3 & 3 & 3 \\
3 & 3 & 3 & 3 & 3 & 3 & 3 & 3 & 3 & 3 & 3 & 4 & 4 & 4 & 4 & 4 & 4 \\
4 & 4 & 4 & 4 & 4 & 4 & 4 & 4 & 4 & 4 & 4 & 5 & 5 & 5 & 5 & 6\\
5 & 5 & 5 & 5 & 5 & 5 & 6 & 6 & 6 & 6 & 6 & 6 & 6\\
6 & 6 & 6 & 6 & 6
\end{ytableau}
\end{equation*}
Let $\mu^i$ denote the partition obtained by restricting to the boxes filled with entries that are most $i$. Note that, if $w_i(\mu)=n-1=17$, then the partition $\mu^{i-1}$ determines the 17 boxes filled with the entry $i$. Thus, the sequence of partitions $\mu^2\subset\mu^5\subset\mu^6=\mu$ completely determines the filling of $\mu$ with type $(17,11,17,17,10,13)$.

We now describe a re-filling of the above shape in such a way that $w_i(\mu)=n-1-(a_{i}-a_{i-1})$ for $i=1,\ldots,6$, that is, $\mu$ instead has type $(15,13,16,15,13,13)$, in such a way that preserves the partitions $\mu^2\subset\mu^5\subset\mu^6=\mu$.

First, note that the shape $\mu^2=(17,11)$ consists of $n-1-(a_{s_1}-a_{s_0})=11$ columns with $s_1-s_0=2$ boxes and the remaining six columns have one box. The 11 columns with two boxes must be filled downward with the entries $1,2$. Of the remaining six columns, four must contain the entry 1 (or equivalently, be missing the entry 2) and two must contain the entry 2 (be missing the entry 1). We fill these six columns from left to right, first the four missing the entry 2, then the remaining two missing the entry 1. The resulting filling of $\mu^2$ is
\begin{equation*}
\begin{ytableau}
1 & 1 & 1 & 1 & 1 & 1 & 1 & 1 & 1 & 1 & 1 & 1 & 1 & 1 & 1 & 2 & 2 \\
2 & 2 & 2 & 2 & 2 & 2 & 2 & 2 & 2 & 2 & 2 
\end{ytableau}
\end{equation*}

We next fill the skew shape $\mu^5/\mu^2=(17,17,17,15,6)/(17,11)$. It consists of $n-1-(a_{s_2}-a_{s_1})=10$ columns with $s_2-s_0=3$ boxes and seven remaining columns have two boxes each. Those with three boxes must be filled downward with the entries $3,4,5$. Of the remaining seven columns, four must be missing the entry 5, two must be missing the entry 4, and one must be missing the entry 3. We place these columns this order from left to right, and fill each in the unique possible way, that is, with entries increasing down columns. The resulting filling of $\mu^5$ (when combined with that of $\mu^2$ is:
\begin{equation*}
\begin{ytableau}
1 & 1 & 1 & 1 & 1 & 1 & 1 & 1 & 1 & 1 & 1 & 1 & 1 & 1 & 1 & 2 & 2 \\
2 & 2 & 2 & 2 & 2 & 2 & 2 & 2 & 2 & 2 & 2 & 3 & 3 & 3 & 3 & \textbf{3} & \textbf{4} \\
3 & 3 & 3 & 3 & 3 & 3 &  \textbf{3} &  \textbf{3} & \textbf{3} &  \textbf{3} &  \textbf{3} & 4 & 4 & 4 & 4 & \textbf{5} & \textbf{5} \\
4 & 4 & 4 & 4 & 4 & 4 & \textbf{4}  & \textbf{4}  & \textbf{4}  & \textbf{4}  & \textbf{5}  & 5 & 5 & 5 & 5\\
5 & 5 & 5 & 5 & 5 & 5
\end{ytableau}
\end{equation*}
The entries of the columns of $\mu^5/\mu^2$ containing only two entries are shown in bold. 

Finally, the skew shape $\mu^6/\mu^5$ is filled in the only way possible, with the entry 6 appearing in all 13 boxes:
\begin{equation*}
\begin{ytableau}
1 & 1 & 1 & 1 & 1 & 1 & 1 & 1 & 1 & 1 & 1 & 1 & 1 & 1 & 1 & 2 & 2 \\
2 & 2 & 2 & 2 & 2 & 2 & 2 & 2 & 2 & 2 & 2 & 3 & 3 & 3 & 3 & \textbf{3} & \textbf{4} \\
3 & 3 & 3 & 3 & 3 & 3 &  \textbf{3} &  \textbf{3} & \textbf{3} &  \textbf{3} &  \textbf{3} & 4 & 4 & 4 & 4 & \textbf{5} & \textbf{5} \\
4 & 4 & 4 & 4 & 4 & 4 & \textbf{4}  & \textbf{4}  & \textbf{4}  & \textbf{4}  & \textbf{5}  & 5 & 5 & 5 & 5\\
5 & 5 & 5 & 5 & 5 & 5 & 6 & 6 & 6 & 6 & 6 & 6 & 6\\
6 & 6 & 6 & 6 & 6
\end{ytableau}
\end{equation*}

We observe in addition that this new SSYT $\mu$ has an $(i,i+1)$-strip for $i=1,3,4$. The $(1,2)$-strip is clearly visible in the top row; we describe how to find $(3,4)$- and $(4,5)$-strips, necessarily inside $\mu^5/\mu^2$. Note that there are two contiguous blocks of columns of $\mu^5/\mu^2$ of size 3. Consider first the leftmost six columns, which are bordered on the right by a column of $\mu^5/\mu^2$ missing the entry 5. Then, we make bold the entries other than 5 in the leftmost six columns. Similarly, in columns 12 through 15, which are bordered on both sides by columns of $\mu^5/\mu^2$ missing the entry 4, we make bold the entries other than 4. We obtain now
\begin{equation*}
\begin{ytableau}
1 & 1 & 1 & 1 & 1 & 1 & 1 & 1 & 1 & 1 & 1 & 1 & 1 & 1 & 1 & 2 & 2 \\
2 & 2 & 2 & 2 & 2 & 2 & 2 & 2 & 2 & 2 & 2 & \textbf{3} &  \textbf{3} & \textbf{3} &  \textbf{3} & \textbf{3} & \textbf{4} \\
\textbf{3} &  \textbf{3} & \textbf{3} &  \textbf{3} &  \textbf{3} & \textbf{3} &  \textbf{3} &  \textbf{3} & \textbf{3} &  \textbf{3} &  \textbf{3} & 4 & 4 & 4 & 4 & \textbf{5} & \textbf{5} \\
 \textbf{4}  & \textbf{4}  & \textbf{4}  & \textbf{4}&  \textbf{4}  & \textbf{4} & \textbf{4}  & \textbf{4}  & \textbf{4}  & \textbf{4}  & \textbf{5}  &  \textbf{5}  &  \textbf{5}  &  \textbf{5}  &  \textbf{5} \\
5 & 5 & 5 & 5 & 5 & 5 & 6 & 6 & 6 & 6 & 6 & 6 & 6\\
6 & 6 & 6 & 6 & 6
\end{ytableau}
\end{equation*}
Now, each column of $\mu^5/\mu^2$ has exactly two entries in bold. The upper bold entries in each column form a $(3,4)$-strip, and the lower bold entries in each column form a $(4,5)$-strip.

The proof below describes an analogous algorithm in general, producing for each term of the product 
\begin{equation*}
 \prod_{j=1}^{\ell}\sigma_{(n-1)^{s_j-s_{j-1}-1},n-1-(a_{s_j}-a_{s_{j-1}})} 
 \end{equation*}
a SSYT of shape $\mu$ with the desired type and strips. We will show in addition that this algorithm gives a bijection between the terms in the product at hand and the needed SSYTs.

\begin{proof}[Proof of Proposition \ref{term_stripless}]
The Pieri rule implies that the terms $\sigma_\mu$ in the expansion of
\begin{equation*}
 \prod_{j=1}^{\ell}\sigma_{(n-1)^{s_j-s_{j-1}-1},n-1-(a_{s_j}-a_{s_{j-1}})} 
 \end{equation*}
 are in bijection with SSYTs of shape $\mu$ with $w_{s_j}(\mu)=n-1-(a_{s_j}-a_{s_{j-1}})$ for $j=1,\ldots,\ell$ and $w_{i}(\mu)=n-1$ for all other values of $i\in\{1,\ldots,r\}$.

Such an SSYT $\mu$ is determined by the data of the sequence of subpartitions
\begin{equation*}
\mu^{s_1}\subset\mu^{s_2}\cdots\subset\mu^{s_\ell}=\mu,
 \end{equation*}
where $\mu^i$ is, as before, the collection of boxes of $\mu$ with entry at most $i$. Furthermore, the skew shape $\mu^{s_j}/\mu^{s_{j-1}}$ has the property that $a_{s_j}-a_{s_{j-1}}$ of its columns have size $s_j-s_{j-1}-1$, while the remaining $n-1-(a_{s_j}-a_{s_{j-1}})$ have size $s_j-s_{j-1}$. We will refer to the former columns as \emph{short} and the latter as \emph{tall}. 

We now describe a filling of the same shape $\mu$, by re-filling each skew-shape $\mu^{s_j}/\mu^{s_{j-1}}$ with the entries $s_{j-1}+1,\ldots,s_j$. The tall columns are filled in the only semi-standard way possible, with the entries $s_{j-1}+1,\ldots,s_j$ from top to bottom. The short columns are filled such that the unique missing entries of each column are non-increasing from left to right, and the entries increase going down each column. The short columns may furthermore be filled in a unique way such that the weight of $i$ is equal to $n-1-(a_{i}-a_{i-1})$ for all $i\in[s_{j-1}+1,s_j]$, because the total number of squares of $\mu$ is
\begin{equation*}
\sum_{i=s_{j-1}+1}^{s_j}[n-1-(a_{i}-a_{i-1})]=(n-1)(s_j-s_{j-1}-1)+[(n-1)-(a_{s_j}-a_{s_{j-1}})].
\end{equation*}

We claim that the new filling is a SSYT; it suffices to restrict our attention to the individual skew-shapes $\mu^{s_j}/\mu^{s_{j-1}}$. The entries are increasing down columns by construction. Between consecutive columns of the same height (tall or short), the entries are weakly increasing across rows if the two columns are flush with each other, and this remains true if the column on the right is shifted upward. If a tall column lies immediately to the left of a short column, the entries are weakly increasing across rows because the highest box of the short column cannot lie below that of the tall column. Similarly, if a short column lies immediately to the left of a tall column, we get the same conclusion because the lowest box of the short column cannot lie above that of the short column.

We next show that the new filling contains the needed 1-strips, by making some of the entries of $\mu^{s_j}/\mu^{s_{j-1}}$ bold. First, all entries in each short column are made bold. Then, for each contiguous block of tall columns (not necessarily all flush with each other), suppose that the short columns to the immediate left and right are missing the entries $p\ge q$, respectively. We take $p=s_j$ if our block of tall columns includes the leftmost column of $\mu$, and $q=s_{j-1}$ if it includes the rightmost column of $\mu$. Now, pick any integer $s\in[q,p]$, and make all entries of the block of the tall columns bold except $s$. In this way, each column of $\mu^{s_j}/\mu^{s_{j-1}}$ contains exactly $s_j-s_{j-1}-1$ bold entries.

We claim that, for $t=1,2,\ldots,s_j-s_{j-1}-1$, taking the $t$-th bold entry in each column of $\mu^{s_j}/\mu^{s_{j-1}}$ gives a $(s_{j-1}+t,s_{j-1}+t+1)$-strip. By construction, these entries must each be equal to one of $s_{j-1}+t,s_{j-1}+t+1$. More precisely, if the missing (in a short column) or unique non-bold (in a tall column) entry is strictly greater than $s_{j-1}+t$, then the $t$-th bold entry is equal to $s_{j-1}+t$, and the $t$-th bold entry is otherwise equal to $s_{j-1}+t+1$. Because the missing and non-bold entries do not increase from left to right, this sequence of bold entries is given by some number of instances of $s_{j-1}+t$, followed by the remaining number of instances of $s_{j-1}+t+1$. (Note that both numbers must be non-zero, because $w_i(\mu)=n-1-(a_i-a_{i-1})<n-1$ for all $i$.) Furthermore, we find that these bold entries do not to move downward upon reading from left to right, by analyzing the possible relative positions of short and tall columns as in the proof that our new filling is a SSYT.

We have therefore produced, from each term $\sigma_\mu$ appearing in
\begin{equation*}
 \prod_{j=1}^{\ell}\sigma_{(n-1)^{s_j-s_{j-1}-1},n-1-(a_{s_j}-a_{s_{j-1}})},
 \end{equation*}
 a SSYT $\mu$ satisfying the needed type and 1-strip properties. 
 
We may also reverse the algorithm. Given a SSYT $\mu$ as in the statement of the proposition, we claim that the sequence of partitions
 \begin{equation*}
\mu^{s_1}\subset\mu^{s_2}\cdots\subset\mu^{s_\ell}=\mu,
 \end{equation*}
has the property that every column of $\mu^{s_j}/\mu^{s_{j-1}}$ has either $s_j-s_{j-1}$ or $s_{j}-s_{j-1}-1$ columns. Because only $s_j-s_{j-1}$ distinct entries are allowed in the filling of $\mu^{s_j}/\mu^{s_{j-1}}$, each column can have no more than this number of boxes. On the other hand, for $t=2,\ldots,s_j-s_{j-1}-1$, a $(s_{j-1}+t-1,s_{j-1}+t)$-strip and a $(s_{j-1}+t,s_{j-1}+t+1)$-strip must be disjoint. If this were not the case, then given any box $b$ in their intersection, necessarily filled with the entry $s_{j-1}+t$, the entry $s_{j-1}+t$ would appear in every column to the right of $b$ (as part of the $(s_{j-1}+t-1,s_{j-1}+t)$-strip) and every column to the left of $b$ (as part of the $(s_{j-1}+t,s_{j-1}+t+1)$-strip). However, by assumption, we cannot have $w_{s_{j-1}+t}(\mu)=n-1$, a contradiction. Thus, each column of $\mu^{s_j}/\mu^{s_{j-1}}$ must contain at least $s_j-s_{j-1}-1$ boxes, coming from the same number of pairwise disjoint strips.

We conclude that the skew-shape $\mu^{s_j}/\mu^{s_{j-1}}$ may be re-filled in a unique semi-standard way with $n-1$ instances each of the entries $s_{j-1}+1,\ldots,s_j-1$, and $n-1-(a_{s_j}-a_{s_{j-1}})$ instances of the entry $s_j$. Combining these fillings yields a SSYT of shape $\mu$ corresponding to a term of 
\begin{equation*}
 \prod_{j=1}^{\ell}\sigma_{(n-1)^{s_j-s_{j-1}-1},n-1-(a_{s_j}-a_{s_{j-1}})}.
 \end{equation*}
 The associations we have constructed are easily seen to be inverse bijections, completing the proof.
\end{proof}

\begin{cor}\label{pie_refined}
In the expansion of $M(a_1,\ldots,a_{r-1};n)$ in terms of the basis of Schubert cycles, the coefficient of $\sigma_\mu$ (with $\mu\subset(n-1)^{r}$ and $|\mu|=(n-1)(r-1)$ as before) is equal to the number of 1-strip-less SSYTs of shape $\mu$ with $i$-weight $n-1-(a_{i}-a_{i-1})$.
\end{cor}

\begin{proof}
As we have already seen, SSYTs with the prescribed number of entries are automatically 0-strip-less, because $n-1-(a_{i}-a_{i-1})<n-1$. Now, the claim is immediate from Corollary \ref{explicit_M}, Proposition \ref{term_stripless}, and Inclusion-Exclusion.
\end{proof}

\begin{rem}\label{degree_shift}
Recall that $M(a_1,\ldots,a_{r-1};n)$ is defined as the pushforward of $\sigma_\lambda\sigma_{\widetilde{\lambda}}$ from $\Gr(r,n)$ to $\Gr(r,n+r-1)$. Thus, all Schubert classes $\sigma_\mu$ appearing in the expansion of $M(a_1,\ldots,a_{r-1};n)$ have $\mu_r\ge r-1$, whereas the same is not true of the individual terms studied in Proposition \ref{term_stripless}. In particular, any SSYT of shape $\mu$ surviving in the sum of the right hand side of Corollary \ref{explicit_M} contains the entries $1,2,\ldots,r$ exactly once in each of the first $r-1$ columns. Thus, the Schubert coefficients of $M(a_1,\ldots,a_{r-1};n)$ may be regarded just as well as the number of 1-strip-less SSYTs of shape $\mu-(r-1)^r$ with $i$-weight $n-r-(a_{i}-a_{i-1})$.
\end{rem}

\begin{cor}\label{pie_all}
In the expansion of 
\begin{equation*}
\sum_{\lambda\subset(n-r-1)^{r-1}}\sigma_{\lambda}\sigma_{\widetilde{\lambda}}\in H^{2(r-1)(n-r-1)}(\Gr(r,n))
\end{equation*}
in terms of the basis of Schubert cycles, the coefficient of $\sigma_\mu$ (now with $\mu\subset(n-r)^{r}$ and $|\mu|=(n-r-1)(r-1)$) is equal to the number of 1-strip-less SSYTs of shape $\mu$.
\end{cor}

Let us again give an illustrative example of the proof. We take $r=3$ and $n=7$, so we sum over partitions $\lambda\subset 3^2$. In the below table, the partitions $\lambda,\widetilde{\lambda}$ are shown, along with the corresponding pairs $(a_1,a_2)$ for which $1=a_0<a_1<a_2<a_3=7$, and the required type $(w_1,w_2,w_3)=(5-a_1,4-(a_2-a_1),a_2-3)$, as prescribed by Corollary \ref{pie_refined} and Remark \ref{degree_shift}, of a 1-strip-less SSYT counted in the expansion of $\sigma_{\lambda}\sigma_{\widetilde{\lambda}}$. 
\begin{center}
\begin{tabular}{ c | c | c | c }
$\lambda$ & $\widetilde{\lambda}$ & $(a_1,a_2)$ & $(w_1,w_2,w_3)$ \\
\hline
$\emptyset$ & $(3,3)$ & $(5,6)$ & $(0,3,3)$ \\
$(1)$ & $(3,2)$ & $(4,6)$ & $(1,2,3)$\\
$(2)$ & $(3,1)$ & $(3,6)$ & $(2,1,3)$\\
$(3)$ & $(3)$ & $(2,6)$ & $(3,0,3)$\\
$(1,1)$ & $(2,2)$ & $(4,5)$ & $(1,3,2)$\\ 
$(2,1)$ & $(2,1)$ & $(3,5)$ & $(2,2,2)$\\
$(3,1)$ & $(2)$ & $(2,5)$ & $(3,1,2)$\\
$(2,2)$ & $(1,1)$ & $(3,4)$ & $(2,3,1)$\\
$(3,2)$ & $(1)$ & $(2,4)$ & $(3,2,1)$\\
$(3,3)$ & $\emptyset$ & $(2,3)$ & $(3,3,0)$
\end{tabular}
\end{center}

In the right-most column, we see all possible types of a 1-strip-less SSYT of size 6 inside the partition $(4)^3$, as no entry can appear $n-r=4$ times without forming a $0$-strip. Thus, it follows from Corollary \ref{pie_refined} that summing over all $\sigma_{\lambda}\sigma_{\widetilde{\lambda}}$ yields precisely a sum of Schubert cycles corresponding to 1-strip-less SSYTs.

\begin{proof}[Proof of Corollary \ref{pie_all}]
Summing over $\lambda\subset(n-r-1)^{r-1}$ amounts to summing over integers $a_i$ with $1<a_1<\cdots<a_{r-1}<n$, and the coefficients appearing in the class in question are the sums of the coefficients appearing in Corollary \ref{pie_refined}, where the partitions $\mu$ are according to Remark \ref{degree_shift}. Varying over all possible $a_i$ has the effect of varying over all distributions of the entries $1,2,\ldots,r$ in a SSYT of shape $\mu$, except that $n-r-(a_{i}-a_{i-1})$ may not be equal to $n-r$. However, the presence of $n-r$ (the largest possible) instances of a given entry $i$ amounts to the existence of a $(i)$-strip, and is therefore already excluded. The conclusion follows.
\end{proof}

\section{Comparison to Klyachko's formula}\label{klyachko_comparison_sec}

\begin{prop}\label{coeffs_fill0}
The coefficient of $\sigma_\mu$ in 
\begin{equation*}
\sum_{1<a_1<\cdots<a_{r-1}<n}M(a_1,\ldots,a_{r-1};n)\in H^{2(r-1)(n-1)}(\Gr(r,n+r-1))
\end{equation*}
is equal to 
\begin{equation*}
\sum_{j=0}^{m(\mu)}(-1)^j\binom{n-r+j-1}{j}|\ssyt^0_{r-j}(\mu^j)|,
\end{equation*}
where we recall that $|\ssyt^0_{r-j}(\mu')|$ is the number of 0-strip-less SSYTs of shape $\mu'$. We also adopt here the notation of Theorem \ref{bf_formula} for the integers $m(\mu)$ and partitions $\mu^j$, which are now defined in reference to parts of the maximal length $n-1$, rather than $n-r$.
\end{prop}

\begin{proof}
For a given $\mu$, fix $\ell\le m(\mu)$. It suffices to identify the number $\binom{n-\ell-1}{r-\ell}|\ssyt^0_\ell(\mu^{r-\ell})|$ as the coefficient of $\sigma_\mu$ in the sum 
\begin{equation*}
\sum_{0=s_0<s_1<\cdots< s_\ell =r}\prod_{j=1}^{\ell}\sigma_{(n-1)^{s_j-s_{j-1}-1},n-1-(a_{s_j}-a_{s_{j-1}})}
\end{equation*}
appearing in Corollary \ref{explicit_M}, where the value of $\ell$ is fixed (and equal to the given $\ell$).

Consider the set of fillings counted by $|\ssyt^0_\ell(\mu^{r-\ell})|$, that is, 0-strip-less SSYTs of shape $\mu^{r-\ell}$ with entries $1,2,\ldots,\ell$. For a given such filling, let $\alpha_j$ denote the number of instances of the entry $j$; the 0-strip-lessness amounts to the fact that $\alpha_j\le n-2$. On the other hand, the coefficient of $\sigma_\mu$ in the product
\begin{equation*}
\prod_{j=1}^{\ell}\sigma_{(n-1)^{s_j-s_{j-1}-1},n-1-(a_{s_j}-a_{s_{j-1}})}=(\sigma_{n-1})^{r-\ell}\prod_{j=1}^{\ell}\sigma_{n-1-(a_{s_j}-a_{s_{j-1}})}
\end{equation*}
counts, when written in terms of the Schubert basis, (necessarily 0-strip-less) SSYTs of shape $\mu^{r-\ell}$ with $\alpha_j=n-1-(a_{s_j}-a_{s_{j-1}})$. 

Therefore, given a 0-strip-less SSYT of shape $\mu^{r-\ell}$, write
\begin{equation*}
a_{s_j}=jn-(\alpha_1+\cdots+\alpha_j)-(j-1).
\end{equation*}
for $j=0,1,2,\ldots,\ell$, so that $a_{s_0}=a_1=1$ and $a_{s_\ell}=a_r=n$. Because $\alpha_j\le n-2$, the $a_{s_j}$ are strictly increasing, and in particular, pairwise distinct. Of the remaining $n-\ell-1$ integers among $1,\ldots,n$, a choice of $r-\ell$ of them corresponds to a choice of a set of integers $\{a_0,a_1,\ldots,a_{r-1},a_r\}\subset\{1,2,\ldots,n\}$ containing the values $a_{s_j}$ already fixed, and furthermore, this choice determines the indices $s_j$, as we require $1<a_1<\cdots<a_{r-1}<n$. With the values of $s_j$ now determined, the SSYT we started with now corresponds naturally to a term of $(\sigma_{n-1})^{r-\ell}\prod_{j=1}^{\ell}\sigma_{n-1-(a_{s_j}-a_{s_{j-1}})}$.

We have therefore described a map from a set of cardinality $\binom{n-\ell-1}{r-\ell}|\ssyt^0_\ell(\mu^{r-\ell})|$ to the set of terms $\sigma_\mu$ appearing in the sum
\begin{equation*}
\sum_{0=s_0<s_1<\cdots< s_\ell =r}\prod_{j=1}^{\ell}\sigma_{(n-1)^{s_j-s_{j-1}-1},n-1-(a_{s_j}-a_{s_{j-1}})}
\end{equation*}
after expanding; it is routine to check that this map is a bijection. This completes the proof.
\end{proof}

We are now ready to complete the proofs of the main theorems.

\begin{proof}[Proof of Theorem \ref{main_formula}]
Combining Proposition \ref{coeffs_fill0} with Lemma \ref{count_0strip}, the coefficient of $\sigma_\mu$ in 
\begin{equation*}
\sum_{1<a_1<\cdots<a_{r-1}<n}M(a_1,\ldots,a_{r-1};n)\in H^{2(r-1)(n-1)}(\Gr(r,n+r-1))
\end{equation*}
is 
\begin{align*}
&\sum_{j=0}^{m(\mu)}(-1)^j\binom{n-r+j-1}{j}\left(\sum_{k=0}^{m(\mu)-j}(-1)^k\binom{r-j}{k}|\ssyt_{r-j-k}(\mu^{j+k})|\right)\\
=&\sum_{\ell=0}^{m(\mu)}(-1)^{\ell}\binom{n}{\ell}|\ssyt_{r-\ell}(\mu^\ell)|,
\end{align*}
where we have used the identity
\begin{equation*}
\sum_{j=0}^{\ell}\binom{n-r+j-1}{j}\binom{r-j}{\ell-j}=\binom{n}{\ell}.
\end{equation*}

The Theorem now follows from the fact that this coefficient is also equal to the coefficient of $\sigma_{\mu-(r-1)^r}$ in $\sum_{\lambda\subset(n-r-1)^{r-1}}\sigma_{\lambda}\sigma_{\widetilde{\lambda}}$.
\end{proof}

\begin{proof}[Proof of Theorem \ref{stripless_interpretation}]
We have seen in the previous proof that $\Gamma^\mu_{r,n}$, as defined in Theorem \ref{k_formula}, is equal to the coefficient of $\sigma_\mu$ in the sum of the $M(a_1,\ldots,a_{r-1};n)$. By Corollary \ref{pie_all}, this coefficient is also equal to the number of 1-strip-less SSYTs of shape $\mu$.
\end{proof}

\appendix
\section{1-strip-less tableau and SYT with $r-1$ descents}

Recall that a \emph{standard Young tableau (SYT)} of shape $\lambda$ is a filling of the boxes of $\lambda$ with the entries $1,2,\ldots,|\lambda|$, each entry appearing exactly once, so that entries increase strictly across rows and down columns. The entry $i$ is a \emph{descent} of the filling if the entry $i+1$ appears in a strictly lower row. Below, the SYT of shape $\lambda=(10,9,4,2)$ has 3 descents: 4, 13, and 22.
\begin{equation*}
\begin{ytableau}
1 & 2 & 3 & 4 & 8 & 9 & 10 & 11 & 12 & 13 \\
5 & 6 & 7 & 17 & 18 & 19 & 20 & 21 & 22\\
14 & 15 & 16 & 25\\
23 & 24
\end{ytableau}
\end{equation*}

Klyachko showed in \cite{k2} that the coefficients $\Gamma^{\mu}_{r,n}$ count SYTs of shape $\overline{\mu}$ with exactly $r-1$ descents. Having shown that $\Gamma^{\mu}_{r,n}$ also count 1-strip-less SSYTs of shape $\mu$, we have:

\begin{prop}\label{bijection}
Fix $r,n$ as in the main body of the paper. Then, the number of SYTs of shape $\overline{\mu}$ with exactly $r-1$ descents is equal to the number of 1-strip-less SSYTs of shape $\mu$.
\end{prop}

We describe a bijection between these two sets of objects, thereby giving a self-contained proof of Proposition \ref{bijection}. We thank Philippe Nadeau for explaining this bijection in private communication.

As in the main body of the paper, we fix integers $r<n$. Consider a SYT of shape $\overline{\mu}$ with descents $i_1<\cdots<i_{r-1}$. Then, we construct a new filling of the Young diagram of $\overline{\mu}$ in the following way: replace the entries $1,\ldots,i_1$ with the entry 1, the entries $i_1+1,\ldots,i_2$ with the entry 2, and so on, until the entries greater than $i_{r-1}$ are replaced with the entry $r$. We refer to this new filling of the same shape as $\overline{\mu}'$ to avoid confusion. If $\overline{\mu}$ is the filling of the partition $\lambda=(10,9,4,2)$ at the beginning of this section, then $\overline{\mu}'$ is the filling
\begin{equation*}
\begin{ytableau}
1 & 1 & 1 & 1 & 2 & 2 & 2 & 2 & 2 & 2 \\
2 & 2 & 2 & 3 & 3 & 3 & 3 & 3 & 3\\
3 & 3 & 3 & 4\\
4 & 4
\end{ytableau}
\end{equation*}

We claim that the new filling $\overline{\mu}'$ is an SSYT. Indeed, if the boxes of $\overline{\mu}'$ are filled in order according to our construction, then the fact that $\overline{\mu}$ is an SYT implies that, at each step, the box being filled is the leftmost in its row and highest in its column that has not yet been filled. Thus, the entries of $\overline{\mu}'$ increase weakly both across rows and down columns. Furthermore, in the sequence of boxes filled with a given entry $j$ in $\overline{\mu}'$, each subsequent box must lie strictly to the right of the previous one, because $\overline{\mu}$ has no descents between $i_{j-1}+1$ and $i_j$. Thus, the entries of $\overline{\mu}'$ in fact decrease strictly down columns.

Note further that $\overline{\mu}'$ contains each of the entries $1,2,\ldots,r$ at least once. Also, for $j=1,2,\ldots,r-1$, the rightmost appearance of the entry $j$ in $\overline{\mu}'$, say in box $b_j$ (corresponding to the last in the sequence of boxes filled with $j$), cannot appear strictly to the left of the leftmost appearance of the entry $j+1$, say in box $b_{j+1}$ (corresponding to the first in the sequence of boxes filled with $j+1$). Indeed, if this were the case, then box $b_{j+1}$ must also appear strictly below box $b_j$, corresponding to a descent in $\overline{\mu}$. On the other hand, the unique box $b$ in the same column as $b_j$ and the same row as $b_{j+1}$ could only be filled with the entry $j+1$, contradicting the assumption that the leftmost appearance of the entry $j+1$ appears in box $b_{j+1}$.

Next, take the unique filling of shape $\mu=\overline{\overline{\mu}'}$ with the property that, when $\mu,\overline{\mu}'$ are drawn complementary to each other inside a $r\times(n-r)$-rectangle $R$, with $\mu$ top- and left-justified, and $\overline{\mu}'$ (now rotated 180 degrees) is filled as above, each column contains each of the entries $1,2,\ldots,r$ exactly once, and the entries of $\mu$ decrease down rows. In the example above, this procedure gives the filling of the full $r\times(n-r)$ rectangle
\begin{equation*}
\begin{ytableau}
1 & 1 & 1 & 1 & 1 & 1 & 2 & 4 & 4 & 4 \\
3 & 4 & 4 & 4 & 4 & 4 & 4 & 3 & 3 & 3\\
4 & 3 & 3 & 3 & 3 & 3 & 3 & 2 & 2 & 2\\
2 & 2 & 2 & 2 & 2 & 2 & 1 & 1 & 1 & 1
\end{ytableau}
\end{equation*}
and the resulting filling $\mu$ is 
\begin{equation*}
\begin{ytableau}
1 & 1 & 1 & 1 & 1 & 1 & 2 & 4\\
3 & 4 & 4 & 4 & 4 & 4\\
4
\end{ytableau}
\end{equation*}

We claim that the filling $\mu$ is also a SSYT. It is strictly decreasing down rows by construction. To show that it is weakly increasing across rows, consider a pair of columns $c_1$ and $c_2$ of $R$, with $c_1$ lying to the left. If we fill these two columns of $R$, at the $s$-th step adding the entry $s$ to each column according to where it appears in either $\mu$ or $\overline{\mu}'$, then the fact that the filling of $\overline{\mu}'$ is semi-standard implies that, at every step, there are at least as many filled boxes of $c_2$ in $\overline{\mu}'$ as there are filled boxes of $c_1$ in $\overline{\mu}'$. Thus, there are at least as many filled boxes of $c_1$ in $\mu$ as there are filled boxes of $c_2$ in $\mu$. This, in turn, implies that, in each row, the entry of $\mu$ in $c_1$ is at most that in $c_2$, as needed.

We claim further that the filling of $\mu$ is 1-strip-less. Note first that $\mu$ cannot contain a 0-strip, because $\overline{\mu}'$ contains each entry $1,2,\ldots,r$ at least once. If $\mu$ had a $(j,j+1)$-strip, in which the entry $j$ appears in the first $m$ columns of $\mu$ and the entry $j+1$ appears in the last $n-r-m$ columns of $\mu$, then $\overline{\mu}'$ (in the standard orientation) would contain no instances of the entry $j$ in its last $m$ columns, and no instances of the entry $j+1$ in its first $n-r-m$ columns. This would mean that the rightmost instance of the entry $j$ appears strictly to the left of the leftmost instance of the entry $j+1$, which contradicts our observation above.

We have therefore produced, from a SYT of shape $\overline{\mu}$, a 1-strip-less SSYT of shape $\mu$. We may now reverse the procedure.

Given a 1-strip-less SSYT of shape $\mu$, we obtain a complementary SSYT $\overline{\mu}'$ by the same method as above, filling the rectangle $R$. That $\mu$ is 1-strip-less implies that $\overline{\mu}'$ contains each entry $1,2,\ldots,r$ at least once, and, for $j=1,2,\ldots,r-1$, that the rightmost instance of the entry $j$ in $\overline{\mu}'$ does not appear strictly to the left of the leftmost instance of the entry $j+1$ in $\overline{\mu}'$. This in turn implies, in $\overline{\mu}'$, that the rightmost instance of the entry $j$ appears strictly above the leftmost instance of the entry $j+1$.

We then rotate $\overline{\mu}'$ into its standard orientation. Ordering the boxes containing the entry 1 from left to right, then the boxes containing the entry 2 from left to right, and so on, we form a filling $\overline{\mu}$ of the same shape by filling the $j$-th box with the entry $j$. This filling $\overline{\mu}$ is a SYT. Indeed, the fact that the entries of $\overline{\mu}'$ strictly decrease down columns immediately implies the same of $\overline{\mu}$, and the fact that the entries of $\overline{\mu}'$ weakly increase across rows and that we order boxes containing the same entry from left to right implies that the entries of $\overline{\mu}$ increase across rows.

By construction, descents of $\overline{\mu}$ can only occur when passing from a box containing the entry $j$ in $\overline{\mu}'$ to one containing the entry $j+1$, because, when moving from left to right through the boxes of $\overline{\mu}'$ containing the same entry $j$, one can only move upward. On the other hand, the rightmost instance of the entry $j$ in $\overline{\mu}'$ appears strictly to the left of the leftmost instance of the entry $j+1$ in $\overline{\mu}'$, so passing between the corresponding two boxes of $\overline{\mu}$ yields a descent for each $j=1,2,\ldots,r-1$. Thus, the SYT $\overline{\mu}$ has exactly $r-1$ descents.

Finally, it is straightforward to see that the two associations we have described are inverse bijections between the objects in question.


\begin{thebibliography}{AMSa}



\bibitem[BF17]{bf} Andrew Berget and Alex Fink. \emph{Equivariant Chow classes of matrix orbit closures},
Transform. Groups \textbf{22} (2017), 631–643.

\bibitem[Cos09]{coskun} Izzet Coskun, \emph{A Littlewood–Richardson rule for two-step flag varieties}, Invent. Math. \textbf{176}, (2009), 325-395.

\bibitem[Kly85]{k} A. A. Klyachko, \emph{Orbits of a maximal torus on a flag space}, Funktsional. Anal. i
Prilozhen. \textbf{19} (1985), 77–78.

\bibitem[Kly95]{k2} A. A. Klyachko, \emph{Toric varieties and flag spaces}. Trudy Mat. Inst. Steklov. \textbf{208} (1995), 139–162. 

\bibitem[LPST20]{lpst} Mitchell Lee, Anand Patel, Hunter Spink, and Dennis Tseng, \emph{Orbits in $(\bP^r)^n$ and equivariant quantum cohomology}, Adv. Math. \textbf{362} (2020).

\bibitem[L23]{l_coll} Carl Lian, \emph{Degenerations of complete collineations and geometric Tevelev degrees of $\bP^r$}, arXiv 2308.00046.

\bibitem[NT23]{nt} Philippe Nadeau and Vasu Tewari, \emph{The permutahedral variety, mixed Eulerian numbers, and principal specializations of Schubert polynomials}, Int. Math. Res. Not. IMRN (2023), 3615–3670.

\bibitem[Spe09]{speyer} David E Speyer. \emph{A matroid invariant via the K-theory of the Grassmannian},
Adv. Math. \textbf{221} (2009), 882–913.

\bibitem[Sta99]{EC2} Richard Stanley, Enumerative Combinatorics, Vol. 2, Cambridge University Press, 1999.



\end{thebibliography}
\end{document}